\documentclass[11pt]{article}
\usepackage{indentfirst}
\usepackage{amssymb,latexsym,amsmath,amsbsy,amsthm,amsxtra,amsgen,graphicx,makeidx,dsfont,longtable,bbm}
\usepackage{tikz}

\newtheorem{remark}{Remark}[section]
\newtheorem{corollary}{Corollary}[section]
\newtheorem{lemma}{Lemma}[section]
\newtheorem{proposition}{Proposition}[section]
\newtheorem{theorem}{Theorem}[section]

\def\Var{\textup{Var}}
\def\Cov{\textup{Cov}}

\def\R{\mathbb{R}}

\def\E{\mathbb{E}}
\def\P{\mathbb{P}}

\title{Transition in the ancestral reproduction rate and its implications for the site frequency spectrum}
\author{Yubo Shuai\\ University of California San Diego}

\begin{document}

\maketitle


\begin{abstract}
Consider a supercritical birth and death process where the children acquire mutations. We study the mutation rates along the ancestral lineages in a sample of size $n$ from the population at time $T$. The mutation rate is time-inhomogenous and has a natural probabilistic interpretation. We use these results to obtain asymptotic results for the site frequency spectrum associated with the sample.
\end{abstract}

\section{Introduction}
Consider a growing population with mutations introduced at the times of birth events. While it is a common assumption in population genetics that mutations are introduced in a time-homogeneous manner \cite{durrett2013population}, the mutations observed in the sample have a time-inhomogenous behavior. This comes from the fact that birth events occurring at different times have different chances of being observed in the sample. In this paper, we analyze the mutations observed in the sample and apply this result to the site frequency spectrum.  

\subsection{Ancestral reproduction rate}
Cheek and Johnston \cite{cheek2023ancestral} considered a continuous-time Galton-Watson process where each individual reproduces at rate $r$. In a reproduction event, an individual is replaced by $k\ge 0$ individuals with probability $p_k$. Denote by $N_t$ the population size at time $t$ and $F_t(s)=\E[s^{N_t}]$ for $s\in[0,1]$ its generating function. Conditional on survival up to time $T$, i.e., $N_T>0$, an individual is sampled uniformly at random from the population at time $T$. They studied the rate of reproduction events along the ancestral lineage of the sampled individual. The rate is biased towards large reproduction events and is inhomogeneous along the ancestral lineage. Specifically, they show the following theorem.
\begin{theorem}[Theorem 2.2 of \cite{cheek2023ancestral}]
There exists a random variable $S$ with density $F'_T(s)/(1-F_T(0))$ on [0,1] such that conditional on $S=s$, independently for each $l$, size $l$ reproduction events occur along the uniform ancestral lineage according to a time inhomogeneous Poisson point process with intensity function
$$
r_l(s,t) = rlp_l F_{T-t}(s)^{l-1}.
$$
\end{theorem}
Igelbrink and Ischebeck \cite{igelbrink2023ancestral} generalized this result to the case when the reproduction rates are age-dependent. 

In general, one can also take a sample of size $n$ from the population. When we trace back the ancestral lineages in the sample, they will coalesce when individuals find their common ancestors, giving us a \emph{coalescent tree}. A central question in population genetics is to understand this coalescent tree and the distribution of the mutations along the tree. 

In this paper, we consider sampling $n$ individuals uniformly at random from a supercritical birth and death process with birth rate $\lambda$ and death rate $\mu$, where $\lambda>\mu$. We write $r=\lambda-\mu> 0$ for the net growth rate. When there is a birth event, we assume the number of mutations that the child acquires has a Poisson distribution with mean $\nu$. The process starts with one individual and runs for $T$ units of time. As usual, we write $N_t$ for the population size at time $t$ and $F_t(s)=\E[s^{N_t}]$ for $s\in[0,1]$ for its generating function. Conditional on $N_T\ge n$, we take a uniform sample of size $n$. 

We give the construction of the coalescent tree and the mutations in the following proposition, with the explanation given in Section \ref{Section: mutations in a sample from a population}. It is easier to construct the tree backward in time. That is, we trace back the lineages of the sampled individuals and track the time for individuals to find their common ancestor. If there are $n$ individuals in the sample, then there are $n-1$ coalescent times $H_{i,n, T}$ for $1\le i\le n-1$. We should emphasize that the construction of the coalescent tree without mutations is obtained by Harris, Johnston, and Roberts \cite{harris2020coalescent} and Lambert \cite{lambert2018coalescent} using two different approaches. The approach described in steps 1 and 2 of construction below comes from Lambert \cite{lambert2018coalescent}.
\begin{proposition}\label{Proposition: construction}
One can generate the coalescent tree with the mutations as follows. 
\begin{enumerate}
\item Take the sampling probability $Y_{n,T}$ from the density 
\begin{equation}\label{Mutations in a sample from a population, uniform(n) sample, density for Y}
f_{Y_{n,T}}(y)=\frac{n\delta_T y^{n-1}}{(y+\delta_T-y\delta_T)^{n+1}},\qquad y\in(0,1), 
\end{equation}
where $\delta_t = r/(\lambda e^{rt}-\mu)$.
\item Conditional on $Y_{n,T}=y$, the $0$th branch length $H_{0,n,T}$ is $T$, and independently for $i=1,2,\dots, n-1$, the $i$th branch length $H_{i,n,T}$ has density
\begin{equation}\label{Introduction, density of H_i,n,T conditional on Y_n,T}
f_{H_{i,n,T}|Y_{n,T}=y}=\frac{y\lambda+(r-y\lambda)e^{-rT}}{y\lambda(1-e^{-rT})} \cdot \frac{y\lambda r^2 e^{-rt}}{(y\lambda+(r-y\lambda)e^{-rt})^2}\mathbf{1}_{\{0< t<T\}}.
\end{equation}
We draw vertical lines of lengths  $H_{i,n, T}$ for $i=0,1,2,\dots, n-1$. Then, we draw horizontal lines to the left, stopping when we hit a vertical line. See Figure \ref{Fig: generating the coalescent tree} for an illustration.

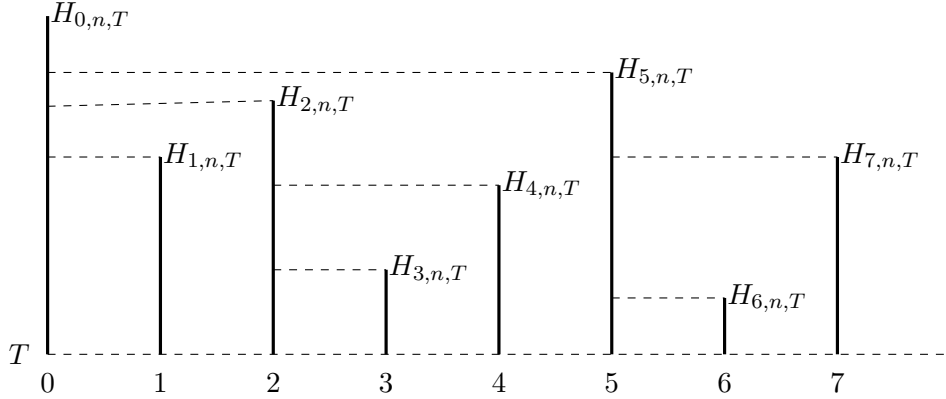
\begin{figure}[h]
\centering
\begin{tikzpicture}[scale=0.75]

\draw[dashed](0,0)--(16,0);
\node at (-0.5,0) {$T$};

\draw[very thick](0,0)--(0,6);
\draw[very thick](2,0)--(2,3.5);
\draw[very thick](4,0)--(4,4.5);
\draw[very thick](6,0)--(6,1.5);
\draw[very thick](8,0)--(8,3);
\draw[very thick](10,0)--(10,5);
\draw[very thick](12,0)--(12,1);
\draw[very thick](14,0)--(14,3.5);

\node at (0.75,6){$H_{0,n,T}$};
\node at (2.75,3.5){$H_{1,n,T}$};
\node at (4.75,4.5){$H_{2,n,T}$};
\node at (6.75,1.5){$H_{3,n,T}$};
\node at (8.75,3) {$H_{4,n,T}$};
\node at (10.75,5){$H_{5,n,T}$};
\node at (12.75,1){$H_{6,n,T}$};
\node at (14.75,3.5){$H_{7,n,T}$};

\draw[dashed] (0,3.5)--(2,3.5);
\draw[dashed] (0,4.4)--(4,4.5);
\draw[dashed] (4,1.5)--(6,1.5);
\draw[dashed] (4,3)--(8,3);
\draw[dashed] (0,5)--(10,5);
\draw[dashed] (10,1)--(12,1);
\draw[dashed] (10,3.5)--(14,3.5);

\node at (0, -0.5){$0$};
\node at (2, -0.5){$1$};
\node at (4, -0.5){$2$};
\node at (6, -0.5){$3$};
\node at (8, -0.5){$4$};
\node at (10, -0.5){$5$};
\node at (12, -0.5){$6$};
\node at (14, -0.5){$7$};
\end{tikzpicture}

\caption{A coalescent tree with sample size $n=8$ and 7 coalescent times.}
\label{Fig: generating the coalescent tree}
\end{figure}

\item Mutational events occur at the coalescent times $H_{1,n,T},\dots, H_{n-1,n,T}$ and along the branches. Conditional on $Y_{n, T}=y$, mutational events along the branches occur independently along the $i$th branch at times of an inhomogeneous Poisson process for $i=0,1,\dots,n-1$. At time $t$ time units before sampling, the rate is $\lambda F_t(1-y)$, where
$$
F_t(1-y)=1-\delta_t e^{rt}+\frac{\delta_t^2 e^{rt}(1-y)}{\delta_t+y-y\delta_t},\qquad \delta_t=\frac{r}{\lambda e^{rt}-\mu}.
$$ 
At each mutational event, the number of mutations has a Poisson distribution with mean $\nu$.
\end{enumerate}
\end{proposition}

We are interested in the case when the sample size is much smaller than the population size, so the sampling probability $Y_{n, T}$ is close to 0. To interpret the time-inhomogeneity, $F_t (1-y)$ is the probability conditional on $Y_{n, T}=y$, that an individual born $t$ time units before the sampling will have no descendants alive in the sample. As $t\rightarrow0$, we have $F_t(1-y)\rightarrow 1-y\approx 1$, which is the probability that the newborn is not sampled. As $t\rightarrow\infty$, we have $F_t(1-y)\rightarrow \mu/\lambda<1$, which is the probability that the subtree generated by the newborn goes extinct. Therefore, we see more mutations along the branches near the bottom part of the tree.

To compare our result with that of Cheek and Johnston \cite{cheek2023ancestral}, observe that the density function of $Y_{n,T}$ with $n=1$ is
$$
f_{Y_{1,T}}(y) = \frac{\delta_T}{(y+\delta_T-y\delta_T)^2}=\frac{r(\lambda e^{rT}-\mu)}{(y\lambda e^{rT}-y\lambda+r)^2}.
$$
Using the formula for $F_t(\cdot)$, one can compute the density function of the random variable $S$ in \cite{cheek2023ancestral} to be
$$
\frac{F_T'(s)}{1-F_T(0)} = \frac{r(\lambda e^{rT}-\mu)}{\left((1-s)\lambda e^{rT}-(1-s)\lambda+r\right)^2}.
$$
Therefore, the random variable $S$ agrees with $1-Y_{1,T}$. To interpret $r_l(s,t)$ in \cite{cheek2023ancestral} correctly, note that the total rate of birth and death events is $\lambda+\mu$ and the probability of a size two reproduction event, i.e., a birth event, is $p_2=\lambda/(\lambda+\mu)$. Therefore, we have
$$
r_2(s,t) = (\lambda+\mu)\cdot 2\cdot\frac{\lambda}{\lambda+\mu} F_{T-t}(s)=2\lambda F_{T-t}(s).
$$
The extra constant 2 comes from the fact that only the children get mutations in our model, and the discrepancy between $T-t$ and $t$ is because time goes forward in \cite{cheek2023ancestral} and backward in Proposition \ref{Proposition: construction}.

\subsection{The site frequency spectrum}
A commonly used summary statistic for mutational data is the \emph{site frequency spectrum}. The site frequency spectrum for $n$ individuals at time $T$ consists of $(M^1_{n,T}, M^2_{n,T}, \dots, M^{n-1}_{n,T})$, where $M^i_{n,T}$ is the number of mutations inherited by $i$ individuals. Durrett \cite{durrett2013population} considered a supercritical birth and death process where mutations occur at a constant rate along the branches with rate $\nu$. He considered sampling the entire population and obtained that as the sampling time $T\rightarrow\infty$, the number of mutations affecting at least a fraction $x$ of the population is approximately
$\nu/\lambda x$. Similar results also appear in \cite{bozic2016quantifying,williams2016identification}. Durrett \cite{durrett2013population} also considered taking a sample of size $n$ and showed that as the sampling time $T\rightarrow\infty$, for $2\le k\le n-1$
\begin{equation}\label{Introduction: Durrett}
\lim_{T\rightarrow\infty} \E[M^{k}_{n,T}] = \frac{n\nu}{k(k-1)}.
\end{equation} 
Gunnarsson, Leder, and Foo \cite{gunnarsson2021exact} considered models where the mutations occur at the times of reproduction events or at a constant rate along the branches. In both models, they computed the exact site frequency spectrum when the sample consists of the entire population or the individuals with an infinite line of descendants, at both a fixed time and a random time when the population size reaches a fixed level. Lambert \cite{lambert2010contour} and Lambert and Stadler \cite{lambert2013birth} developed a framework to study the site frequency spectrum when each individual is sampled independently with probability $p$, conditioned to have $n$ individuals. Following this approach, Delaporte, Achaz, and Lambert \cite{delaporte2016mutational} computed $\E[M^k_{n, T}]$ for critical birth and death process, for both fixed $T$ and for an unknown $T$ with some improper prior. Dinh et al. \cite{dihn2020statistical} computed $\E[M^{k}_{n, T}]$ for supercritical birth and death process. As noted by Ignatieva, Hein, and Jenkins \cite{ignatieva2020characterisation}, if we take the sampling probability $p\rightarrow0$ in \cite{dihn2020statistical}, then one can recover the $1/k(k-1)$ shape in the site frequency spectrum as in formula \eqref{Introduction: Durrett}. 

Beyond the expected value of $M^k_{n, T}$, there are other works on its asymptotic behavior. Gunnarsson, Leder, and Zhang \cite{gunnarsson2023limit} considered a supercritical Galton-Watson process and sampled the entire population. They proved a strong law of large numbers for the site frequency spectrum at deterministic times and a weak law of large numbers at random times when the population size reaches a deterministic level. Schweinsberg and Shuai \cite{schweinsberg2023asymptotics} considered a sample of size $n$ from supercritical or critical birth and death processes. As the sample size and the population size go to infinity appropriately, they established the asymptotic normality for the length of the branches that support $k$ leaves. These results can then be translated into the asymptotic normality of the site frequency spectrum, assuming that mutations occur at a constant rate along the branches.

In this paper, since we assume mutations occur at the times of reproduction events, the site frequency spectrum is characterized by the number of reproduction events. Indeed, let $R^k_{n, T}$ (resp. $R^{\ge 2}_{n, T}$) be the number of
reproduction events in which the child has $k$ (resp. at least 2) descendants in the sample. Then conditional on $(R^1_{n,T}, R^2_{n,T}, \dots, R^{n-1}_{n,T})$, $M^i_{n,T}$ has a Poisson distribution with mean $\nu R^i_{n,T}$, and the $M^i_{n,T}$ are independent of one another. Therefore, we will focus on $R^k_{n,T}$ and $R^{\ge 2}_{n,T}$. We show the following theorem.
\begin{theorem}\label{Theorem: LLN}
Let $\{T_n\}_{n=1}^\infty$ be a sequence such that
\begin{equation}\label{Introduction: sequence condition, LLN}   
\lim_{n\rightarrow\infty} ne^{-rT_n} = 0.
\end{equation}
Then for any fixed $k\ge 2$
$$
\frac{R^k_{n,T_n}}{n} \stackrel{P}{\longrightarrow} \frac{\lambda}{rk(k-1)},
$$
where ``$\stackrel{P}{\longrightarrow}$" denotes convergence in probability as $n\rightarrow\infty$.
\end{theorem}

\begin{corollary}\label{Corollary: number of mutations inherited by k individuals}
Under the same condition as Theorem \ref{Theorem: LLN}, we have
$$
\frac{M^k_{n,T_n}}{n}\stackrel{P}{\longrightarrow} \frac{\lambda \nu}{rk(k-1)}.
$$
\end{corollary}

Under a stronger assumption on the sample size and the population size, we can establish a central limit theorem for $R^{\ge 2}_{n, T_n}$. The corresponding $M^{\ge 2}_{n, T_n}=\sum_{i=1}^{n-1}M^{i}_{n, T_n}$ represents the number of mutations shared by at least two individuals in the sample. The quantity $M^{\ge 2}_{n, T_n}$ has been applied to estimate the growth rate of a tumor in \cite{johnson2023clonerate}. We believe an analogous result should hold for $R^{k}_{n, T_n}$, but the covariance computation is more involved.
\begin{theorem}\label{Theorem}
Let $\{T_n\}_{n=1}^{\infty}$ be a sequence such that 
\begin{equation}\label{Introduction: sequence condition}
\lim_{n\rightarrow\infty} n^{3/2}(\log n) e^{-rT_n} = 0.
\end{equation}
Then
$$
\frac{1}{\sqrt{n}}\left(R^{\ge 2}_{n,T_n}-\frac{n\lambda}{r}\right) \Rightarrow N\left(0, \frac{\lambda^2}{r^2}\right),
$$
where $``\Rightarrow"$ denotes convergence in distribution as $n\rightarrow\infty$.
\end{theorem}

\begin{corollary}\label{Corollary: number of mutations}
Under the same condition as Theorem \ref{Theorem}, we have
$$
\frac{1}{\sqrt{n}}\left(M^{\ge 2}_{n,T_n}-\frac{n\lambda\nu}{r}\right)\Rightarrow N\left(0,\frac{\lambda^2\nu^2}{r^2}+\frac{\lambda\nu}{r}\right).
$$
\end{corollary} 
\begin{remark}
Results similar to Theorem \ref{Theorem: LLN} and Corollary \ref{Corollary: number of mutations} have been established when the mutations occur at a constant rate $\widetilde{\nu}$ along the branches. We take $\widetilde{\nu}=\lambda \nu$ so that the expected number of mutations per unit of time is the same under both models. Schweinsberg and Shuai \cite{schweinsberg2023asymptotics} showed the following result, which is similar to Theorem \ref{Theorem: LLN}.
\end{remark}
\begin{theorem}[Theorem 1.3 of \cite{schweinsberg2023asymptotics}]
Let $k\ge 2$ be a fixed integer. Assume that 
\begin{equation*}
\lim_{n\rightarrow\infty}n e^{-r_n T_n}=0.
\end{equation*} 
Write $L^{k}_{n, T_n}$ for the total length of the branches that support $k$ leaves in the coalescent tree of the sample. Then as $n\rightarrow\infty$, 
$$
\frac{r_n}{n}L_{n,T_n}^k\stackrel{P}{\longrightarrow} \frac{1}{k(k-1)}.
$$
\end{theorem}
The hypothesis is similar to \eqref{Introduction: sequence condition, LLN}, and the $1/k(k-1)$ site frequency spectrum is consistent with \eqref{Introduction: Durrett}. Conditional on $L^{k}_{n,T_n}$, the number of mutations inherited by $k$ individuals has a Poisson distribution with mean $\widetilde{\nu}L^{k}_{n,T_n}\approx n\widetilde{\nu}/(r_nk(k-1))=n\lambda \nu/(r_nk(k-1))$, consistent with Theorem \ref{Theorem: LLN}.

Johnson et al. showed a result similar to Corollary \ref{Corollary: number of mutations}. The quantity $M^{in}_n$ therein corresponds to $M^{\ge 2}_{n,T_n}$. In the special case where the rates $\lambda,\mu$ and $\widetilde{\nu}$ are constants, Corollary 2 of \cite{johnson2023clonerate} reduces to the following theorem.
\begin{theorem}[Corollary 2 of \cite{johnson2023clonerate}]
Assume that \eqref{Introduction: sequence condition} holds. Then
$$\frac{1}{\sqrt{n}} \bigg( M_{n,T_n}^{\ge 2} - \frac{n \widetilde{\nu}}{r} \bigg) \Rightarrow N\left(0,\frac{\widetilde{\nu}}{r^2}+\frac{\widetilde{\nu}}{r}\right).$$
\end{theorem}
The hypothesis is the same as \eqref{Introduction: sequence condition}, and the result is the same when we plug in $\widetilde{\nu}=\lambda\nu$. Therefore, our results establish that the assumption that mutations occur at a constant rate along the branches is not essential to the results in \cite{johnson2023clonerate, schweinsberg2023asymptotics}, and similar results hold when mutations occur only at the birth times, which may be a more natural assumption for some applications.

Another closely related summary statistic for the mutational data is the \emph{allele frequency spectrum}. For the allele frequency spectrum, the individuals are partitioned into families where each family consists of individuals carrying the same mutations. Then the allele frequency spectrum consists of $(A^1_{n,T}, A^2_{n,T}, \dots, A^n_{n,T})$ where $A^i_{n,T}$ is the number of families of size $i$. Richard \cite{richard2014splitting} considered the allele frequency spectrum of the population in a birth and death process with age-dependent death rates and mutations at birth times. He computed the expected allele frequency spectrum and proved a law of large numbers for the allele frequency spectrum. Champagnat and Lambert considered the case where the death rate is age-dependent and the mutations occur at a constant rate along the branches. They computed the expected allele frequency spectrum when the sample consists of the entire population in \cite{champagnat2012splitting} and studied the size of the largest family and the age of the oldest family in \cite{champagnat2013splitting}.

\section{The contour process of the genealogical tree}
Figure \ref{Fig: the planar representation of a genealogical tree} is the planar representation of a genealogical tree. The ends of the vertical lines are the birth and death times for individuals, and the horizontal lines represent reproduction events. The individuals are ordered from left to right in the following way. An individual $x$ is on the left of $y$ if either $x$ is ancestral to $y$, or when we trace back the lineages of $x$ and $y$ to their common ancestor, the ancestral lineage of $y$ belongs to the child. For example, in Figure \ref{Fig: the planar representation of a genealogical tree}, individual $6$ is on the left of $7$ because $6$ is an ancestor of $7$, and $7$ is on the left of $8$ because the as we trace back the ancestral lineage of $7$ (the red line) and $8$ (the blue line), the ancestral lineage of $8$ belongs to the child.
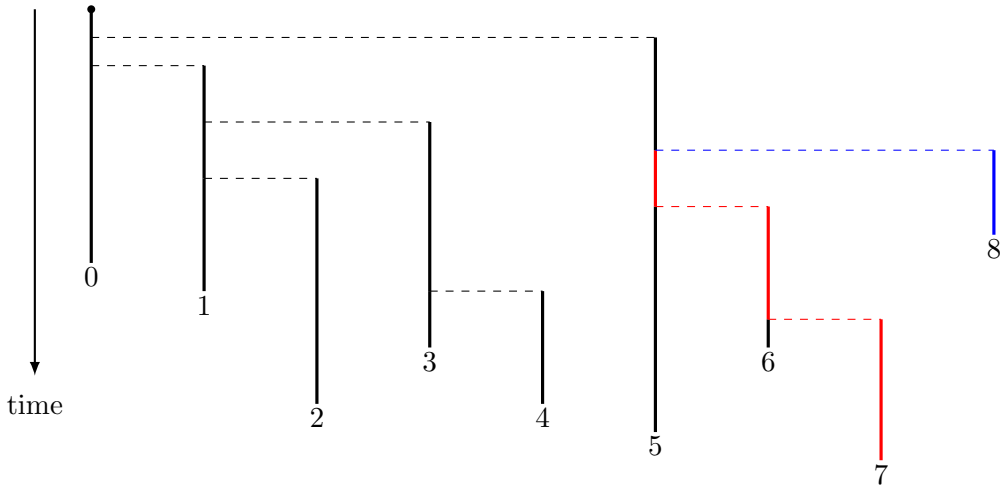
\begin{figure}[h]
\centering
\begin{tikzpicture}[scale=0.75]

\fill (0,6.5) circle(2pt);

\draw[thick, -latex](-1, 6.5)--(-1,0);
\node at (-1, -0.5){time};

\draw[very thick](0,2)--(0,6.5);
\draw[very thick](2,1.5)--(2,3.5);
\draw[very thick](2,3.5)--(2,5.5);
\draw[very thick](4,-0.5)--(4,3.5);
\draw[very thick](6,0.5)--(6,1.5);
\draw[very thick](6,1.5)--(6,4.5);
\draw[very thick](8,-0.5)--(8,1.5);
\draw[very thick](10,-1)--(10,3);
\draw[very thick, red](10,3)--(10,4);
\draw[very thick](10,4)--(10,6);
\draw[very thick](12,0.5)--(12,1);
\draw[very thick, red](12,1)--(12,3);
\draw[very thick, red](14,-1.5)--(14,1);
\draw[very thick, blue](16,2.5)--(16,4);

\node at (0, 1.75){$0$};
\node at (2, 1.25){$1$};
\node at (4, -0.75){$2$};
\node at (6, 0.25){$3$};
\node at (8, -0.75){$4$};
\node at (10, -1.25){$5$};
\node at (12, 0.25){$6$};
\node at (14, -1.75){$7$};
\node at (16, 2.25){$8$};

\draw[dashed] (0,5.5)--(2,5.5);
\draw[dashed] (2,3.5)--(4,3.5);
\draw[dashed] (2,4.5)--(6,4.5);
\draw[dashed] (6,1.5)--(8,1.5);
\draw[dashed] (0,6)--(10,6);
\draw[dashed, red] (10,3)--(12,3);
\draw[dashed, red] (12,1)--(14,1);
\draw[dashed, blue] (10,4)--(16,4);
\end{tikzpicture}

\caption{The planar representation of a genealogical tree}
\label{Fig: the planar representation of a genealogical tree}
\end{figure}

The contour process of a genealogical tree is a c\`adl\`ag process which encodes the information of the tree. This process starts at the depth of the leftmost individual, keeps track of the distance from the root in the depth-first search of the genealogical tree, and is indexed by the total length of the searched branches. See Figure \ref{Fig: contour process} for an illustration. This process has slope $-1$ as we trace back the lineage in the genealogical tree unless we encounter a new birth event, in which case the contour process makes a jump in the size of the child's lifespan. 
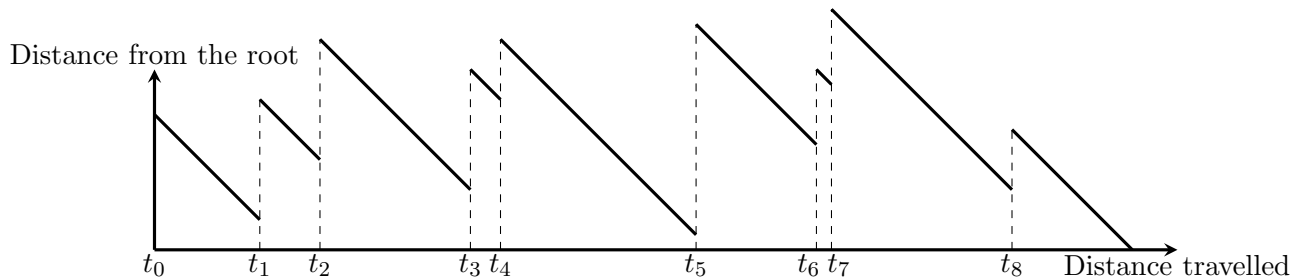
\begin{figure}[h]
\centering

\begin{tikzpicture}[scale=0.4]

\draw[very thick, -stealth](0,0)--(34,0);
\node at (34,-0.5){Distance travelled};

\draw[very thick, -stealth](0,0)--(0,6);
\node at (0,6.5){Distance from the root};

\draw[very thick](0,4.5)--(3.5,1);
\draw[dashed](3.5,0)--(3.5,5);
\draw[very thick](3.5,5)--(5.5,3);
\draw[dashed](5.5,0)--(5.5,7);
\draw[very thick](5.5,7)--(10.5, 2);
\draw[dashed](10.5,0)--(10.5,6);
\draw[very thick](10.5,6)--(11.5,5);
\draw[dashed](11.5,0)--(11.5, 7);
\draw[very thick](11.5,7)--(18,0.5);
\draw[dashed](18, 0)--(18,7.5);
\draw[very thick](18, 7.5)--(22,3.5);
\draw[dashed](22, 0)--(22,6);
\draw[very thick](22,6)--(22.5, 5.5);
\draw[dashed](22.5, 0)--(22.5,8);
\draw[very thick](22.5, 8)--(28.5,2);
\draw[dashed](28.5, 0)--(28.5,4);
\draw[very thick](28.5, 4)--(32.5,0);

\node at (0, -0.5){$t_0$};
\node at (3.5, -0.5){$t_1$};
\node at (5.5, -0.5){$t_2$};
\node at (10.5, -0.5){$t_3$};
\node at (11.5, -0.5){$t_4$};
\node at (18, -0.5){$t_5$};
\node at (21.75, -0.5){$t_6$};
\node at (22.75, -0.5){$t_7$};
\node at (28.5, -0.5){$t_8$};


\end{tikzpicture}
\caption{The contour process of the genealogical tree in Figure \ref{Fig: the planar representation of a genealogical tree}. For $i=1,2,...,8$, $t_i$ is the total length of the branches searched before we reach the $i$th individual, and the size of the jump is the lifespan of the $i$th individual.}
\label{Fig: contour process}
\end{figure}

One can also study the tree truncated at time $T$ and its contour process. We denote the contour processes for the original and truncated trees by $X$ and $X^{(T)}$, respectively. Lambert \cite{lambert2010contour} showed that $X$ (resp. $X^{(T)}$) has the same distribution as $Y$ (resp. $Y^{(T)}$) stopped at 0 defined as follows. Let $\{\xi_i\}_{i=0}^{\infty}$ be independent exponentially distributed random variables with mean $1/\mu$ representing the lifespan of individuals, and let $N$ be a Poisson process with rate $\lambda$ counting the number of reproduction events found in the depth-first search. We define Lévy processes $Y$ and $Y^{(T)}$ by
\begin{equation*}
    Y_t = \xi_0-t+\sum_{i=0}^ {N(t)} \xi_{i},\qquad
    Y^{(T)}_t = \min\{\xi_0, T\}-t+\sum_{i=1}^ {N(t)} \min\{T-Y^{(T)}_{t-}, \xi_{i}\}.
\end{equation*}
That is, the process $Y$ starts from the lifespan of the $0$th individual, has drift $-1$ as we trace back the lineage, and makes jumps when we encounter a new birth event. The process $Y^{(T)}$ has the same dynamics except that the jumps are truncated so that $Y^{(T)}$ does not go above $T$.

\section{Mutations in a sample from a population}\label{Section: mutations in a sample from a population}
In most applications, we only have access to the genetic information in a sample of the population at time $T$. Therefore, we are interested in the coalescent tree that describes the genealogy of the sampled individuals. See Figure \ref{Fig: the coalescent tree} for an illustration. The mutations in the sample can be classified into two groups depending on whether they are associated with branching events in the coalescent tree. Those associated with such branching points are marked red in Figure \ref{Fig: the coalescent tree}. The others are marked blue.
\begin{figure}[h]
\centering
\begin{tikzpicture}[scale=0.75]

\draw[dashed](0,0)--(18,0);
\node at (-0.5,0) {$T$};

\draw[very thick](0,2)--(0,5.5);
\draw[very thick, red](0,5.5)--(0,6.5);
\draw[very thick](2,1.5)--(2,3.5);
\draw[very thick,red](2,3.5)--(2,5.5);
\draw[very thick, red](4,0)--(4,3.5);
\draw[very thick](4,-0.5)--(4,0);
\draw[very thick](6,0.5)--(6,4.5);
\draw[very thick](8,-1)--(8,1.5);
\draw[very thick](10,-1)--(10,0);
\draw[very thick, red](10,0)--(10,6);
\draw[very thick](12,0.5)--(12,1);
\draw[very thick, red](12,1)--(12,3);
\draw[very thick](14,-1.5)--(14,0);
\draw[very thick, red](14,0)--(14,1);
\draw[very thick](16,2.5)--(16,4);

\draw[dashed, red] (0,5.5)--(2,5.5);
\draw[dashed, red] (2,3.5)--(4,3.5);
\draw[dashed] (2,4.5)--(6,4.5);
\draw[dashed] (6,1.5)--(8,1.5);
\draw[dashed, red] (0,6)--(10,6);
\draw[dashed, red] (10,3)--(12,3);
\draw[dashed, red] (12,1)--(14,1);
\draw[dashed] (10,4)--(16,4);

\node at (0, 1.75){$0$};
\node at (2, 1.25){$1$};
\node at (4, -0.75){$2$};
\node at (6, 0.25){$3$};
\node at (8, -1.25){$4$};
\node at (10, -1.25){$5$};
\node at (12, 0.25){$6$};
\node at (14, -1.75){$7$};
\node at (16, 2.25){$8$};

\fill [blue](2,5.5) circle(3pt); 
\fill [blue](4,3.5) circle(3pt); 
\fill [red] (10,6) circle(3pt);
\fill [red] (12,3) circle(3pt);
\fill [blue] (14,1) circle(3pt);

\end{tikzpicture}

\caption{The coalescent tree (in red) where individuals 2, 5, and 7 are sampled. This tree is obtained by tracing back the lineages of the sampled individuals from time $T$.}
\label{Fig: the coalescent tree}
\end{figure}
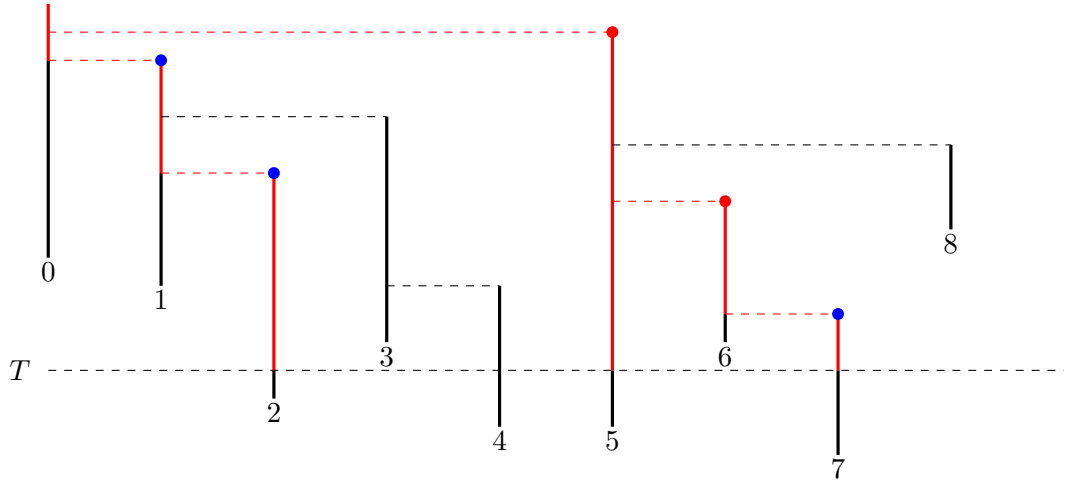

Two sampling schemes have been studied extensively in the literature: sampling each individual independently with probability $y$ and taking a uniform sample of size $n$. We refer to them as Bernoulli$(y)$ sampling and uniform$(n)$ sampling, respectively. These two sampling schemes were connected by Lambert \cite{lambert2018coalescent}. We will discuss the mutations in the Bernoulli$(y)$ sampling in Section \ref{Subsection: Bernoulli(y) sampling} and the uniform$(n)$ sampling in Section \ref{Subsection: uniform$(n)$ sampling}.

In the rest of the paper, we will consider a variant of the model where the parent acquires the blue mutations. By the Markov property, whenever there is a birth event, the subtree generated by the child and its descendants has the same distribution as the subtree generated by the parent and the descendants that it has after the original birth event. Therefore, the distribution of the coalescent tree and the mutations are the same whether we assign the mutations to the parent or the child. For an illustration, one can compare Figures \ref{Fig: the coalescent tree} and \ref{Fig: the coalescent tree with the parent getting the mutations}.

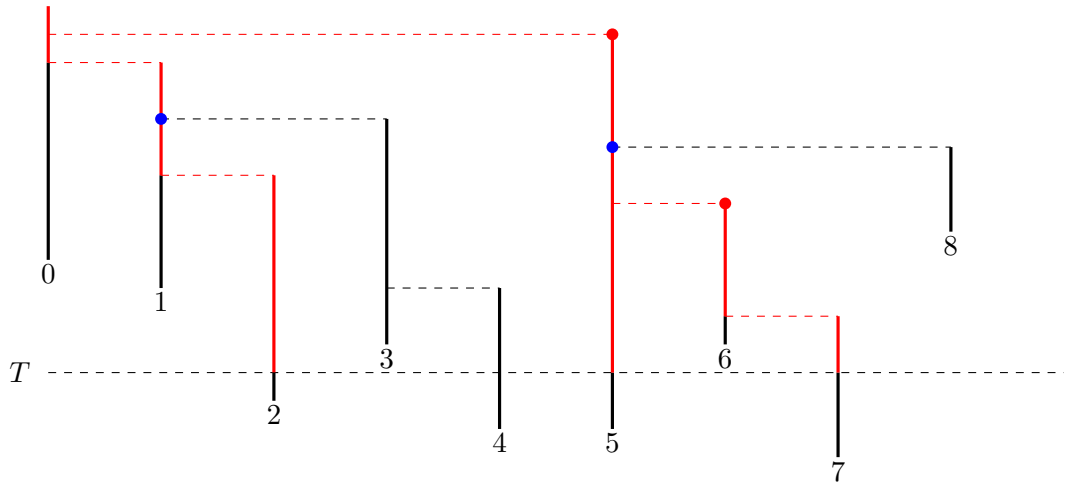
\begin{figure}[h]
\centering
\begin{tikzpicture}[scale=0.75]

\draw[dashed](0,0)--(18,0);
\node at (-0.5,0) {$T$};

\draw[very thick](0,2)--(0,5.5);
\draw[very thick, red](0,5.5)--(0,6.5);
\draw[very thick](2,1.5)--(2,3.5);
\draw[very thick,red](2,3.5)--(2,5.5);
\draw[very thick, red](4,0)--(4,3.5);
\draw[very thick](4,-0.5)--(4,0);
\draw[very thick](6,0.5)--(6,4.5);
\draw[very thick](8,-1)--(8,1.5);
\draw[very thick](10,-1)--(10,0);
\draw[very thick, red](10,0)--(10,6);
\draw[very thick](12,0.5)--(12,1);
\draw[very thick, red](12,1)--(12,3);
\draw[very thick](14,-1.5)--(14,0);
\draw[very thick, red](14,0)--(14,1);
\draw[very thick](16,2.5)--(16,4);

\draw[dashed, red] (0,5.5)--(2,5.5);
\draw[dashed, red] (2,3.5)--(4,3.5);
\draw[dashed] (2,4.5)--(6,4.5);
\draw[dashed] (6,1.5)--(8,1.5);
\draw[dashed, red] (0,6)--(10,6);
\draw[dashed, red] (10,3)--(12,3);
\draw[dashed, red] (12,1)--(14,1);
\draw[dashed] (10,4)--(16,4);

\node at (0, 1.75){$0$};
\node at (2, 1.25){$1$};
\node at (4, -0.75){$2$};
\node at (6, 0.25){$3$};
\node at (8, -1.25){$4$};
\node at (10, -1.25){$5$};
\node at (12, 0.25){$6$};
\node at (14, -1.75){$7$};
\node at (16, 2.25){$8$};

\fill [red] (10,6) circle(3pt); 
\fill [blue](2,4.5) circle(3pt);
\fill [blue] (10,4) circle(3pt);
\fill [red] (12,3) circle(3pt);

\end{tikzpicture}

\caption{The same coalescent tree as in Figure \ref{Fig: the coalescent tree}, except that the parent mutates when there is a birth event that is not at one of the branchpoints in the coalescent tree.}
\label{Fig: the coalescent tree with the parent getting the mutations}
\end{figure}

\subsection{Bernoulli$(y)$ sampling}\label{Subsection: Bernoulli(y) sampling}
In this section, we describe the coalescent tree, together with the red and blue mutations in the Bernoulli sampling scheme. Lambert and Stadler  \cite{lambert2013birth} obtained the construction of the subtree. We outline the construction here for readers' convenience. Recall that the contour process $X^{(T)}$ makes jumps at rate $\lambda$. That is, birth events occur at rate $\lambda$ as we trace back the lineage. Consider a birth event occurring $t$ time units before the sampling time. Recalling the definition of $N_t$ from the introduction, the probability that this birth event does not give rise to a new leaf in the coalescent tree is
\begin{equation}\label{Mutation in a sample from a population, Bernoulli(y) sampling, def of q(y,t)}
q(y,t)=\mathbb{P}\left(N_{t}=0\right)+\sum_{k=1}^\infty \mathbb{P}\left(N_{t}=k\right)(1-y)^k=\E\left[(1-y)^{N_t}\right]=F_t(1-y).   
\end{equation} 
Therefore, red birth events occur with rate $\lambda(1-q(y,t))$, and blue birth events occur with rate $\lambda q(y,t)$, independently of each other. For a coupling argument in Section \ref{Subsection: error bounds}, red and blue birth events are represented using a Poisson point process. To be more precise, for any positive continuous functions $L\le U$ on $[a,b]$, we let $\mathcal{A}(L,U, [a,b])=\{(t,x)\in\R^2: a\le t\le b, L(t)\le x\le U(t)\}$ be the area between $L$ and $U$ on $[a,b]$. In the special case where $L\equiv 0$ on $[a,b]$, we abbreviate the notation as $\mathcal{A}(U, [a,b])$. We also let $\{\Pi_{i}\}_{i\ge 0}$ be independent Poisson point process on $\R^2$ with unit intensity. Then, for the Bernoulli$(y)$ sampling, one can obtain the coalescent tree and the mutations in the following way:
\begin{enumerate}
\item We get an empty sample with probability $\E[(1-y)^{N_T}]=F_T(1-y)$. If the sample is non-empty, then we have an initial branch $H_{0,y, T}$ of length $T$.
\item Let $\{\pi_{i,y}\}_{i\ge0}$ be independent Poisson processes with rate $\lambda (1-q(y,t))$ on $[0,T]$ defined by
$\pi_{i,y}(t)=\Pi_i(\mathcal{A}(\lambda q(y,\cdot), \lambda, [0,t]))$ for $0\le t\le T$.
\item Having constructed $H_{0,y,T},\dots, H_{i,y,T}$, if $\pi_{i,y}(T)>0$, then the $i$th branch length is the time of the first jump in $\pi_{i,y}$, i.e. $H_{i,y, T}=\inf\{t\ge 0: \pi_{i,y}(t)>0\}$. Otherwise, we have found all the individuals in the sample. Note that $H_{i,y, T}$ is also the time for the red reproduction event on the $i$th individual.
\item Blue birth events occur along the branches at rate $\lambda q(y,t)$. That is, the number of mutations along the $i$th branch up to time $t$ is $\Pi_i(\mathcal{A}(\lambda q(y,\cdot), [0,t]))$.
\end{enumerate}
\begin{remark}\label{Remark: population size has a geometric distribution}
This construction implies that conditional on a non-empty sample, the sample size has a geometric distribution. In particular, by taking $y=1$, the population size $N_T$ satisfies $\P(N_T\ge k+1|N_T\ge k)=\P(\pi_{i,1}(T)>0)$ for $k\ge 1$.
\end{remark}

We compute $q(y,t)=F_t(1-y)$ here and refer the reader to equation (7) of \cite{lambert2018coalescent} for the formula for the density of $H_{i,y, T}$. We define 
\begin{equation}\label{Mutation in a sample from a population, Bernoulli(y) sampling, def of delta_t}
\delta_t=\P(\pi_{i,1}(t)=0)=\exp\left(-\int_0^t \lambda \P(N_s>0) ds \right).
\end{equation}
By Remark \ref{Remark: population size has a geometric distribution}, conditional on $N_t>0$, $N_t$ has a geometric distribution with 
$$
\P(N_t\ge k+1|N_t\ge k)=\P(\pi_{i,1}(t)>0)=1-\delta_t,\qquad \text{ for } k\ge 1.
$$ 
Therefore, we have
\begin{equation}\label{Mutations in a sample from a population, Bernoulli(y) sampling, expected population}
e^{rt}=\E[N_t] = \P(N_t>0)/\delta_t.   
\end{equation}
Combining \eqref{Mutation in a sample from a population, Bernoulli(y) sampling, def of delta_t} and \eqref{Mutations in a sample from a population, Bernoulli(y) sampling, expected population}, we get
$$
\delta_t = \exp\left(-\int_0^t \lambda e^{rs}\delta_s ds \right).
$$
Solving this equation for $\delta_t$ and using the initial condition $\delta_{0}=1$, we get the following formula for $\delta_t$.
$$
\delta_{t} = \frac{r}{\lambda e^{rt}-\mu},
$$
which agrees with the formula of $\delta_t$ in \eqref{Mutations in a sample from a population, uniform(n) sample, density for Y}. This formula is also found in equation (4) of \cite{lambert2018coalescent}.
To compute $q(y,t)$, by \eqref{Mutations in a sample from a population, Bernoulli(y) sampling, expected population} and Remark \ref{Remark: population size has a geometric distribution} again, we have
\begin{align}\label{Mutation in a sample from a population, Bernoulli(y) sampling, values of q(y,t) and delta_y,t}
q(y,t)&=F_t(1-y)\nonumber\\ 
&=\P(N_t=0)+\P(N_t\ge1)\E\left[(1-y)^{N_t}|N_t\ge1\right]\nonumber\\
&=1-\delta_t e^{rt}+ \delta_t e^{rt}\cdot \frac{\delta_{t}(1-y)}{\delta_{t}+y-y\delta_{t} }.
\end{align}

\subsection{Uniform$(n)$ sampling}\label{Subsection: uniform$(n)$ sampling}
Lambert \cite{lambert2018coalescent} showed that the uniform$(n)$ sampling can be realized by taking $y$ from the density \eqref{Mutations in a sample from a population, uniform(n) sample, density for Y}
and then performing the Bernoulli($y$) sampling conditioned to have size $n$. Therefore, for the uniform$(n)$ sampling, one can obtain the coalescent tree and the mutations in the following way:
\begin{enumerate}
\item Take the sampling probability $Y_{n,T}$ from the density \eqref{Mutations in a sample from a population, uniform(n) sample, density for Y}.
\item Conditional on $Y_{n,T}=y$, for $i=1,2,\dots, n-1$, let $\pi_{i,y}$ be the Poisson process be defined in Section \ref{Subsection: Bernoulli(y) sampling}, conditional on $\pi_{i,y}(T)>0$. 
\item For $i=1,2,\dots, n-1$, the $i$th branch length $H_{i,n,T}$ is the time of the first jump in $\pi_{i,y}$, conditional on $\pi_{i,y}(T)>0$. That is,
$$
\P(H_{i,n,T}>t|Y_{n,T}=y)=\P(\pi_{i,y}(t)=0|\pi_{i,y}(T)>0).
$$
The density for $H_{i,n, T}$ is given by \eqref{Introduction, density of H_i,n,T conditional on Y_n,T}.
\item Blue birth events occur along the branches at rate $\lambda q(y,t)$.
\end{enumerate}

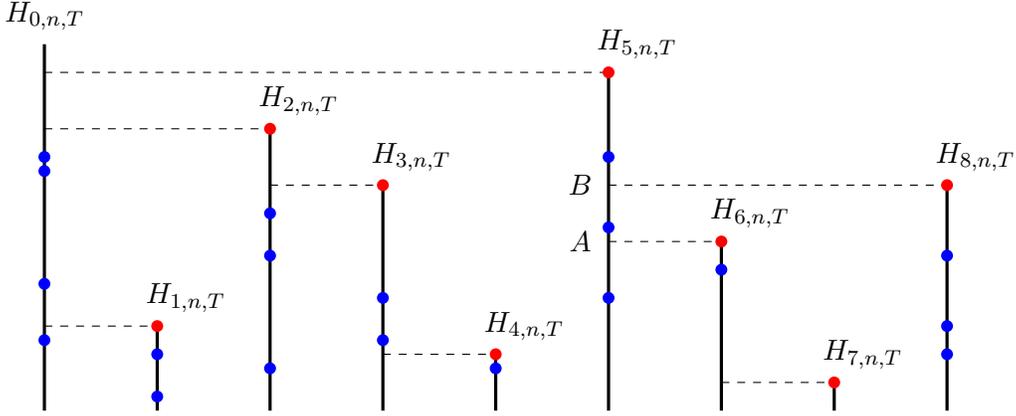
\begin{figure}[h]
\centering
\begin{tikzpicture}[scale=0.75]


\draw[very thick](0,0)--(0,6.5);
\node at (0,7){$H_{0,n,T}$};
\draw[very thick](2,0)--(2,1.5);
\node at (2.5,2){$H_{1,n,T}$};
\draw[very thick](4,0)--(4,5);
\node at (4.5,5.5){$H_{2,n,T}$};
\draw[very thick](6,0)--(6,4);
\node at (6.5,4.5){$H_{3,n,T}$};
\draw[very thick](8,0)--(8,1);
\node at (8.5,1.5){$H_{4,n,T}$};
\draw[very thick](10,0)--(10,6);
\node at (10.5,6.5){$H_{5,n,T}$};
\draw[very thick](12,0)--(12,3);
\node at (12.5,3.5){$H_{6,n,T}$};
\draw[very thick](14,0)--(14,0.5);
\node at (14.5,1){$H_{7,n,T}$};
\draw[very thick](16,0)--(16,4);
\node at (16.5,4.5){$H_{8,n,T}$};

\draw[dashed] (0,1.5)--(2,1.5);
\draw[dashed] (0,5)--(4,5);
\draw[dashed] (4,4)--(6,4);
\draw[dashed] (6,1)--(8,1);
\draw[dashed] (0,6)--(10,6);
\draw[dashed] (10,3)--(12,3);
\draw[dashed] (12,0.5)--(14,0.5);
\draw[dashed] (10,4)--(16,4);

\fill [red] (2,1.5) circle(3pt);
\fill [red] (4,5) circle(3pt);
\fill [red] (6,4) circle(3pt);
\fill [red] (8,1) circle(3pt);
\fill [red] (10,6) circle(3pt);
\fill [red] (12,3) circle(3pt);
\fill [red] (14,0.5) circle(3pt);
\fill [red] (16,4) circle(3pt);

\fill [blue] (0,1.25) circle(3pt);
\fill [blue] (0,2.25) circle(3pt);
\fill [blue] (0,4.25) circle(3pt);
\fill [blue] (0,4.5) circle(3pt);
\fill [blue] (2,0.25) circle(3pt);
\fill [blue] (2,1) circle(3pt);
\fill [blue] (4,0.75) circle(3pt);
\fill [blue] (4,2.75) circle(3pt);
\fill [blue] (4,3.5) circle(3pt);
\fill [blue] (6,1.25) circle(3pt);
\fill [blue] (6,2) circle(3pt);
\fill [blue] (8,0.75) circle(3pt);
\fill [blue] (10,2) circle(3pt);
\fill [blue] (10,3.25) circle(3pt);
\fill [blue] (10,4.5) circle(3pt);
\fill [blue] (12,2.5) circle(3pt);
\fill [blue] (16,1) circle(3pt);
\fill [blue] (16,1.5) circle(3pt);
\fill [blue] (16,2.75) circle(3pt);

\node at (9.5,3){$A$};
\node at (9.5,4){$B$};
\end{tikzpicture}
\caption{A coalescent tree with sample size $n=9$. The blue mutation on the segment $AB$ supports three individuals because $\max\{H_{6,n,T}, H_{7,n,T}\}\le \min\{H_{5,n,T}, H_{8,n,T}\}$, the red mutation associated with $H_{3,n,T}$ supports two individuals because $H_{4,n,T}<H_{3,n,T}<H_{5,n,T}$.}
\label{Fig: counting red and blue mutations}
\end{figure}

\begin{remark}\label{Remark: def of R^k_i,n,T}
For $i=0,1,\dots, n-1$, we define $R^{k}_{i,n, T}$ (resp. $R^{\ge 2}_{i,n, T}$) to be the number of reproduction events along the $i$th branch that support\ $k$ (resp. at least two) leaves in the sample. Then for $k\ge 2$, we have 
$$
R^{k}_{0,n,T} = \Pi_0\left(\mathcal{A}\left(\lambda q(Y_{n,T},\cdot),\left[\max_{1\le j\le k-1}H_{j,n,T}, H_{k,n,T}\right]\right)\right),
$$
and for $1\le i\le n-k-1$,
\begin{align*}
R^{k}_{i,n,T} &= \Pi_i\left(\mathcal{A}\left(\lambda q(Y_{n,T},\cdot),\left[\max_{i+1\le j\le i+k-1} H_{j,n,T},H_{i,n,T}\wedge H_{i+k,n,T}\right]\right)\right)\\
&\qquad+\mathbbm{1}_{\{\max_{i+1\le j\le i+k-1} H_{j,n,T}\le H_{i,n,T}\le H_{i+k,n,T}\}},
\end{align*}
and 
\begin{align*}
R^{k}_{n-k,n,T} &= \Pi_{n-k}\left(\mathcal{A}\left(\lambda q(Y_{n,T},\cdot),\left[\max_{n-k+1\le j\le n-1} H_{j,n,T},H_{n-k,n,T}\right]\right)\right)\\
&\qquad+\mathbbm{1}_{\{\max_{n-k+1\le j\le n-1}H_{j,n,T}\le H_{n-k,n,T}\}}.
\end{align*}
The first part in the definition of $R^{k}_{i,n, T}$ or $R^{k}_{n-k,n, T}$ counts the number of blue reproduction events with $k$ descendants in the sample on the $i$th branch. The second part counts the number of red reproduction events that have $k$ descendants in the sample. See Figure \ref{Fig: counting red and blue mutations} for an illustration. Note that there is no red mutation on the $0$th branch. 
The total number of reproduction events with $k$ descendants in the sample is 
$$
R^{k}_{n,T}=\sum_{i= 0}^{n-k} R^{k}_{i,n,T}.
$$ 
Likewise, the number of reproduction events with at least two descendants in the sample on the $0$th branch is
$$
R^{\ge 2}_{0,n,T} = \Pi_0(\mathcal{A}(\lambda q(Y_{n,T},\cdot),[H_{1,n,T},T])),
$$
and for  $1\le i\le n-2$,
$$
R^{\ge 2}_{i,n,T} = \Pi_i(\mathcal{A}(\lambda q(Y_{n,T},\cdot),[H_{i+1,n,T},H_{i,n,T}]))+\mathbbm{1}_{\{H_{i+1,n,T}\le H_{i,n,T}\}}.
$$
The total number of reproduction events with at least two descendants in the sample is 
$$
R^{\ge 2}_{n,T}=\sum_{i= 0}^{n-2} R^{\ge 2}_{i,n,T}.
$$  
\end{remark}

\section{Asymptotics for large $n$ and $T$}\label{Section: asymptotics for large $n$ and $T$}
\subsection{Approximation for the sampling probability and branch length}\label{Subsection: approximation for the sampling probability and branch length}
Approximations for the sampling probability $Y_{n,T}$ and the branch lengths $\{H_{i,n,T}\}_{i=1}^{n-1}$ when $n$ and $T$ are large are obtained in \cite{johnson2023clonerate}.
\begin{enumerate}
\item Choose $W$ from the exponential density
$$f_{W}(w)=e^{-w},\qquad w\in(0,\infty).$$
\item Choose $U_{i}$ for $i=1,\dots, n-1$ independently from the logistic distribution, with density
$$
f_{U_{i}}(u)=\frac{e^u}{\left(1+e^u\right)^{2}}, \quad u \in(-\infty, \infty).
$$
\item Let $Y = n\delta_T/W$.
\item Let ${H}_{i}=T-(\log n+\log(1/W)+U_i)/r$ for $i=1,\dots, n-1$.
\end{enumerate}
Note that the $H_i$'s (hence the quantities defined below) depend on $n$, although this dependence is not explicit in the notation. Using this approximation, we can, therefore, define an approximation for the number of reproduction events with at least two descendants in the sample on the $i$th branch as 
$$
R^{\ge 2}_{0} = \Pi_0(\mathcal{A}(\lambda q(n\delta_T/ W,\cdot),[H_{1},T])),
$$
and
$$
R^{\ge 2}_{i} = \Pi_i(\mathcal{A}(\lambda q(n\delta_T/ W,\cdot),[H_{i+1},H_{i}]))+\mathbbm{1}_{\{H_{i+1}\le H_{i}\}},\qquad\text{ for } 1\le i\le n-2.
$$
An approximation for the total number of reproduction events with at least two descendants in the sample, excluding the $0$th branch, is 
$$
R^{\ge 2}=\sum_{i= 1}^{n-2} R^{\ge 2}_{i}.
$$   
Similarly, for $k\ge 2$, we define
$$
R^{k}_{0} = \Pi_0\left(\mathcal{A}\left(\lambda q(n\delta_T/ W,\cdot),\left[\max_{1\le j\le k-1}H_{j}, H_{k}\right]\right)\right),
$$
and for $1\le i\le n-k-1$
\begin{align*}
R^{k}_{i} = \Pi_i\left(\mathcal{A}\left(\lambda q(n\delta_T/ W,\cdot),\left[\max_{i+1\le j\le i+k-1} H_{j}, H_{i}\wedge H_{i+k}\right]\right)\right)+\mathbbm{1}_{\{\max_{i+1\le j\le i+k-1} H_{j}\le H_{i}\le H_{i+k}\}},
\end{align*}
and 
\begin{align*}
R^{k}_{n-k} = \Pi_i\left(\mathcal{A}\left(\lambda q(n\delta_T/ W,\cdot),\left[\max_{n-k-1\le j\le n-1} H_j,H_{n-k}\right]\right)\right)+\mathbbm{1}_{\{\max_{n-k-1\le j\le n-1} H_j\le H_{n-k}\}},
\end{align*}
An approximation for the total number of reproduction events inherited by $k$ individuals, excluding the $0$th and the $(n-k)th$ branches, is 
$$
R^{k}=\sum_{i= 1}^{n-k-1} R^{k}_{i}.
$$

\subsection{Error bounds}\label{Subsection: error bounds}
In the rest part of the paper, the sampling time $T=T_n$ depends on $n$. The errors in the approximation resulting from taking the limit as $n\rightarrow\infty$ are bounded by Lemmas 4 and 5 in the 
supplementary material of \cite{johnson2023clonerate}. 
We combine these lemmas into Lemma \ref{Lemma: error in the approximation of Y and H}. The event $A_n$ in Lemma \ref{Lemma: error in the approximation of Y and H} is the intersection of the events $Q_{n,T_n}=Q_{n,\infty}\le h(n):=n^{1/2}e^{rT_n/2}$ and $\widetilde{Q}_{n,\infty}=1/W$ in \cite{johnson2023clonerate}. The event $B_n$ is the intersection of  $Q_{n,T_n}=Q_{n,\infty}\le h(n):=n^{1/2}(\log n)^{-1/3}e^{r_n T_n/3}$ and $\widetilde{Q}_{n,\infty}=1/W$ in \cite{johnson2023clonerate}. The choice of $h(n)$ comes from the proof of Lemma 6 of \cite{johnson2023clonerate}.
\begin{lemma}\label{Lemma: error in the approximation of Y and H}
Assume condition \eqref{Introduction: sequence condition, LLN} holds. The random variables $Y_{n,T_n}$, $W$, $H_{i,n,T_n}$ and $H_i$ can be coupled so that the following hold.
\begin{enumerate}
\item There is a sequence of events $A_n$ on which $Y_{n,T_n}=n\delta_{T_n}/W$ and $\P(A_n)\rightarrow1$ as $n\rightarrow\infty$.
\item As $n\rightarrow\infty$,
$$
\frac{r}{n}\sum_{i=1}^{n-1}\E[|H_{i,n,T_n}-H_i|\mathbbm{1}_{A_n}]\rightarrow0.
$$
\end{enumerate}
Assume condition \eqref{Introduction: sequence condition} holds. The random variables $H_{i,n, T_n}$ and $H_i$ can be coupled so that the following hold.
\begin{enumerate}
\item There is a sequence of events $B_n$ on which $Y_{n,T_n}=n\delta_{T_n}/W$ and $\P(B_n)\rightarrow1$ as $n\rightarrow\infty$.
\item As $n\rightarrow\infty$,
$$
\frac{r}{\sqrt{n}}\sum_{i=1}^{n-1}\E[|H_{i,n,T_n}-H_i|\mathbbm{1}_{B_n}]\rightarrow0.
$$
\end{enumerate}
\end{lemma}

These error bounds allow us to bound the error in the approximations for $R^{\ge 2}_{n, T_n}$ and $R^{k}_{n, T_n}$, with the aid of another technical lemma.
\begin{lemma}\label{Lemma: ordering of coupling}
Let $X_1,X_2,\dots,X_n$ and $Y_1,Y_2,\dots,Y_n$ be two independent sequences of i.i.d random variables such that $P(X_i=X_j)=0$ and $P(Y_i=Y_j)=0$ for all $i\neq j$. Then there exist random variables $\widetilde{Y}_1,\widetilde{Y}_2,\dots, \widetilde{Y}_n$ such that the following hold. 
\begin{enumerate}
    \item The random vectors $(\widetilde{Y}_1,\widetilde{Y}_2,\dots, \widetilde{Y}_n)$ and $(Y_1,Y_2,\dots,Y_n)$ are equal in distribution.
    \item
    For $1\le i,j\le n$, $X_i\le X_j \text{ if and only if } \widetilde{Y}_i\le \widetilde{Y}_j$.
    \item $\sum_{i=1}^{n} |X_i-\widetilde{Y}_{i}|\le\sum_{i=1}^{n} |X_i-Y_i|$, and so $\sum_{i=1}^{n} \E[|X_i-\widetilde{Y}_{i}|]\le\sum_{i=1}^{n} \E[|X_i-Y_i|]$.
\end{enumerate} 
\begin{proof}
    Let $Y_{(1)}\le Y_{(2)}\le\dots\le Y_{(n)}$ be the ordering of the random variables $Y_1,Y_2,\dots, Y_n$. For $1\le i\le n$, we define $\widetilde{Y}_{i}$ to be $Y_{(j)}$ if $X_i$ is the $j$th smallest among $(X_1, X_2, \dots, X_n)$. Then, the second claim is straightforward from the definition. Since the random variables $X_1, X_2,\dots, X_n$ are i.i.d, $\widetilde{Y}_1,\widetilde{Y}_2,\dots, \widetilde{Y}_n$ is therefore a random permutation of the i.i.d random variables $Y_1,Y_2,\dots,Y_n$, which implies the first claim. The last claim follows from the second claim.
\end{proof}
\end{lemma}

\begin{corollary}\label{Corollary: error bounds} Recall the definition of $R^{k}_{i,n,T_n}$ and $R^{\ge 2}_{i,n,T_n}$ from Remark \ref{Remark: def of R^k_i,n,T} and the definition of $R^{k}_{i}$ and $R^{\ge 2}_{i}$ from Section \ref{Subsection: approximation for the sampling probability and branch length}. 
Assume condition \eqref{Introduction: sequence condition, LLN} holds. Under a coupling satisfying the conclusions of  Lemma \ref{Lemma: error in the approximation of Y and H}, we have,
$$
\frac{1}{n} \left(\sum_{i=0}^{n-k}R^{k}_{i,n,T_n}-\sum_{i=1}^{n-k-1}R^{k}_{i}\right) \stackrel{P}{\rightarrow} 0.
$$
Assume condition \eqref{Introduction: sequence condition} holds. Under a coupling satisfying the conclusions of Lemma \ref{Lemma: error in the approximation of Y and H}, we have,
$$
\frac{1}{\sqrt{n}} \left(\sum_{i=0}^{n-2} R^{\ge 2}_{i,n,T_n}-\sum_{i=1}^{n-2}R^{\ge 2}_{i}\right) \stackrel{P}{\rightarrow} 0.
$$
\end{corollary}
\begin{proof}
In view of Lemma \ref{Lemma: ordering of coupling}, we may assume that conditional on $Y_{n, T_n}=y=n\delta_{T_n}/ W$, the random variables $H_i$ and $H_{i,n, T_n}$ are coupled so that $H_i\le H_j$ if and only if $H_{i,n, T_n}\le H_{j,n, T_n}$. Therefore, the only error in the approximation comes from the blue mutations. We write $A\triangle B = (A\setminus B)\cup (B\setminus A)$ and $\mathbf{Leb}$ for the Lebesgue measure on $\R^2$. Conditional on the event $A_y=\{Y_{n,T_n}=y=n\delta_{T_n}/ W\}$, we have
\begin{equation*}
\begin{split}
     \E\left[\left|R^{\ge 2}_{0,n,T_n}\right|\Big|A_y\right]
     &\le \E\left[\left|R^{\ge 2}_{0,n,T_n}-R^{\ge 2}_{0}\right|+R^{\ge 2}_0\Big|A_y\right]\\
     &\le\E\left[\left|R^{\ge 2}_{0,n,T_n}-R^{\ge 2}_{0}\right|\Big|A_y\right]+\lambda\E[T-H_1|A_y]\\
     &=\E\left[\left|R^{\ge 2}_{0,n,T_n}-R^{\ge 2}_{0}\right|\Big|A_y\right]+\frac{\lambda}{r}\E[\log n+\log(1/W)+U_i|A_y],\\
\end{split}
\end{equation*}
and 
\begin{align*}
\E\left[\left|R^{\ge 2}_{0,n,T_n}-R^{\ge 2}_{0}\right|\Big|A_y\right]
&=\E[|\Pi_0(\mathcal{A}(\lambda q(y,\cdot),[H_{1,n,T_n}, T]))-\Pi_0(\mathcal{A}(\lambda q(y,\cdot),[H_{1}, T]))||A_y]\\
&\le \E[\mathbf{Leb}(\mathcal{A}(\lambda q(y,\cdot),[H_{1,n,T_n}, T])\triangle(\mathcal{A}(\lambda q(y,\cdot),[H_{1}, T])))|A_y]\\
&\le \lambda \E[|H_{1, n,T_n}-H_{1}||A_y].   
\end{align*}
For $1\le i\le n-2$, we have 
\begin{align*}
    \E\left[\left|R^{\ge 2}_{i,n,T_n}-R^{\ge 2}_{i}\right|\Big|A_y\right]
    &=\E[|\Pi_i(\mathcal{A}(\lambda q(y,\cdot),[H_{i+1,n,T_n}, H_{i,n,T_n}]))-\Pi_i(\mathcal{A}(\lambda q(y,\cdot),[H_{i+1}, H_{i}]))||A_y]\\
    &\le \E[\mathbf{Leb}(\mathcal{A}(\lambda q(y,\cdot),[H_{i+1,n,T_n}, H_{i,n,T_n}])\triangle(\mathcal{A}(\lambda q(y,\cdot),[H_{i+1}, H_{i}])))|A_y]\\
    &\le \lambda \E[|H_{i+1, n,T_n}-H_{i+1}|+|H_{i,n,T_n}-H_{i}||A_y].
\end{align*} 
The second claim then follows from integrating over the event $B_n$ in Lemma \ref{Lemma: error in the approximation of Y and H}.
The first claim can be proved in an analogous way using the inequality
\begin{align*}
    \E\left[\left|R^{k}_{i,n,T_n}-R^{k}_{i}\right|\Big|A_y\right]\le \lambda \sum_{j=0}^k\E[|H_{i+j, n,T_n}-H_{i+j}||A_y],
\end{align*}
and the observation that $R^{k}_{n-k,n,T_n}$ has the same distribution as $R^{k}_{0,n,T_n}$, and that $R^{k}_{0,n,T_n}\le R^{\ge 2}_{0,n,T_n}$.
\end{proof}

\section{Proof of Theorems \ref{Theorem: LLN} and \ref{Theorem}}\label{Section: proof of theorems}
In this section, we start with Propositions \ref{Proposition: lln} and \ref{Proposition: clt}, which are the counterparts of Theorems \ref{Theorem: LLN} and \ref{Theorem} using the approximations $R^{\ge 2}$ and $R^{k}$. We prove Propositions \ref{Proposition: lln} and \ref{Proposition: clt} in Sections \ref{Subsection: proof of theorem LLN} and \ref{Subsection: proof of theorem}. Theorems \ref{Theorem: LLN} and \ref{Theorem} then follow from Corollary \ref{Corollary: error bounds}. We should emphasize that $R^{\ge 2}$ and $R^{k}$ depend on $n$ and $T_n$, although not made explicit in the notation.

\begin{proposition}\label{Proposition: lln}
Under condition \eqref{Introduction: sequence condition, LLN}, we have the following convergence in probability as $n\rightarrow\infty$
$$
\frac{R^{k}}{n} \stackrel{P}{\longrightarrow} \frac{\lambda}{rk(k-1)}.
$$  
\end{proposition}

\begin{proposition}\label{Proposition: clt}
Under condition \eqref{Introduction: sequence condition}, we have the following convergence in distribution as $n\rightarrow\infty$,
$$
\frac{1}{\sqrt{n}}\left(R^{\ge 2}-\frac{n\lambda}{r}\right) \Rightarrow N\left(0,\frac{\lambda^2}{r^2}\right).
$$
\end{proposition}
To prove Proposition \ref{Proposition: lln} or \ref{Proposition: clt}, it suffices to show that the convergence in probability or distribution holds conditional on $W=w$, regardless of the value of $w$. Consider Proposition \ref{Proposition: clt} for example; we will show that the conditional distribution of $R^{\ge 2}$ given $W=w$ is asymptotically normal with the same mean and variance given by Proposition \ref{Proposition: clt}, regardless of the value of $w$. Proposition \ref{Proposition: clt} then follows from applying the dominated convergence theorem to the characteristic function:
\begin{align*}
&\lim_{n\rightarrow\infty}\E\left[\exp\left(\frac{it}{\sqrt{n}}\left(R^{\ge 2}-\frac{n\lambda}{r}\right)\right)\right] 
=\lim_{n\rightarrow\infty}\E\left[\E\left[\exp\left(\frac{it}{\sqrt{n}}\left(R^{\ge 2}-\frac{n\lambda}{r}\right)\right)\Big| W\right]\right]\\
&\qquad\qquad=\E\left[\lim_{n\rightarrow\infty}\E\left[\exp\left(\frac{it}{\sqrt{n}}\left(R^{\ge 2}-\frac{n\lambda}{r}\right)\right)\Big| W\right]\right]=\exp\left(-\frac{t^2\lambda^2}{2r^2}\right).
\end{align*}

\subsection{Moment computations}
In this section, for positive functions $f,g,h$, we write
$$
f = g(1+O(h))\qquad\text{ if } |f-g|\le Cgh \text{ for some universal constant } C.
$$
We will also use the following fact. For positive integers $0\le m\le n-2$,
\begin{equation}\label{Proof of theorems, moment computation, calculus fact}
\int_{-\infty}^\infty \frac{e^{(m+1)s}}{(1+e^s)^n}\ ds = \int_0^\infty \frac{x^m}{(1+x)^n}\ dx=\sum_{k=0}^m (-1)^{m-k} \frac{1}{n-k-1} {m\choose k}.    
\end{equation}
Another observation is that condition \eqref{Introduction: sequence condition} implies that
\begin{equation}\label{Introduction: sequence condition, equivalent}
\lim_{n\rightarrow\infty} n^{3/2} T_n e^{-rT_n} = 0.
\end{equation}
Indeed, if equation \eqref{Introduction: sequence condition, equivalent} fails, then we may choose a sequence $\{n_k\}$ such that
$$
\lim_{n\rightarrow\infty} n_k^{3/2} T_{n_k} e^{-rT_{n_k}} = c>0.
$$
Combined with \eqref{Introduction: sequence condition}, we can deduce that $\lim_{n\rightarrow\infty} (\log n_k/T_{n_k}) = 0$. However, this implies
$$
\lim_{n\rightarrow\infty} n_k^{3/2} T_{n_k} e^{-rT_{n_k}} = 0,
$$
which is a contradiction.

\begin{lemma}\label{Lemma: moment computation}
We have the following moment estimates for the number of reproduction events with at least two descendants in the sample, excluding those on the $0$th branch. Assume condition \eqref{Introduction: sequence condition, LLN} holds. Then for each fixed $w\in (0,\infty)$, we have
\begin{equation}\label{Proof of theorems, moment computation, R^>=2 variance}
\lim_{n\rightarrow\infty}n^{-1}\Var\left(\sum_{i=1}^{n-2}R^{\ge 2}_{i}\Big|W=w\right)=\frac{\lambda^2}{r^2},    
\end{equation}
\begin{equation}\label{Proof of theorems, moment computation, R^k mean}
\lim_{n\rightarrow\infty}n^{-1}\E\left[\sum_{i=1}^{n-k-1}R^{ k}_{i}\Big|W=w\right]=\frac{\lambda}{rk(k-1)}.  
\end{equation}
Assume condition \eqref{Introduction: sequence condition} holds. Then for each fixed $w\in(0,\infty)$, we have
\begin{equation}\label{Proof of theorems, moment computation, R^>=2 mean}
\lim_{n\rightarrow\infty}\frac{1}{\sqrt{n}}\left|\E\left[\sum_{i=1}^{n-2}R^{\ge 2}_{i}\Big|W=w\right] -\frac{n\lambda}{r}\right|=0.
\end{equation}

\end{lemma}
\begin{proof}
To prove \eqref{Proof of theorems, moment computation, R^>=2 mean}, for $1\le i\le n-2$,  we define
$$
R^{\ge 2, red}_{i} = 
\mathbbm{1}_{\{H_{i+1}\le H_{i}\}},\qquad
R^{\ge 2, blue}_{i}= \Pi_i(\mathcal{A}(\lambda q(n\delta_{T_n}/ W,\cdot),[H_{i+1},H_{i}])), 
$$
to be the approximations for the numbers of red and blue reproduction events on the $i$th branch.
Then we have 
\begin{equation}\label{Proof of theorems, moment computation, red i}
\E\left[R^{\ge 2, red}_{i}\Big| W=w\right] = \frac{1}{2}.
\end{equation}
We now compute $\E[R^{\ge 2, blue}_{i}| W=w]$. Using the change of variable $t=T_n-(s+\log n+\log(1/w))/r$, we have
\begin{align*}
&\E\left[R^{\ge 2, blue}_i\Big|W=w\right]\\ 
&\qquad=\E[\Pi(\mathcal{A}(\lambda q(n\delta_{T_n}/ w, \cdot)), [H_{i+1}, H_{i}])|W=w]\\
&\qquad= \int_0^{T_n} \P(H_{i+1}\le t< H_{i}|W=w) \lambda q(n\delta_{T_n}/ w,t)\ dt\\
&\qquad=\frac{1}{r}\int_{-\log n-\log(1/w)}^{rT_n-\log n-\log(1/w)} \P(U_{i+1}\ge s\ge U_{i}) \lambda q(n\delta_{T_n}/ w,T_n-(s+\log n+\log(1/w))/r) \ ds.
\end{align*}
We consider the limit of $\lambda q(n\delta_{T_n}/ w,T_n-(s+\log n+\log(1/w))/r)$ as $n\rightarrow\infty$. Recall the formula for $q(y,t)$ from \eqref{Mutation in a sample from a population, Bernoulli(y) sampling, values of q(y,t) and delta_y,t}. Writing $t(s)=T_n-(s+\log n+\log (1/w))/r$ and noticing that 
$$
\delta_t = \frac{r}{\lambda}e^{-rt}\left(1+O\left(e^{-rt}\right)\right),
$$
we have,
\begin{align*}
&\lambda q(n\delta_{T_n}/ w, t(s))\\
&\qquad=\lambda\left(1-\delta_{t(s)} e^{rt(s)}+\frac{\delta_{t(s)}(1-n\delta_{T_n} W)}{\delta_{t(s)}+n\delta_{T_n}/ w-\delta_{t(s)}n\delta_{T_n}/ w}\cdot\delta_{t(s)} e^{rt(s)}\right)\\
&\qquad= \lambda \left(1-\frac{r}{\lambda}\left(1+O\left(e^{-rt(s)}\right)\right)+\frac{\delta_{t(s)}}{\delta_{t(s)}+n\delta_{T_n}/ w}\cdot(1+O(n\delta_{T_n}/ w))\cdot\frac{r}{\lambda}\right)\\
&\qquad= \lambda \Big(1-\frac{r}{\lambda}\left(1+O\left(e^{-rt(s)}\right)\right)\\
&\qquad\qquad+\frac{(rn/\lambda w) e^{s-rT_n}}{(rn/\lambda w) e^{s-rT_n}+(rn/\lambda w)e^{-rT_n}}\cdot(1+O(e^{-rt(s)}+n\delta_{T_n}/ w))\cdot\frac{r}{\lambda}\Big)\\
&\qquad=\mu+\frac{r}{1+e^{-s}}+O(re^{-rt(s)}+rn\delta_{T_n}/ w).
\end{align*}
Therefore, using $\P(U_{i+1}\ge s\ge U_{i})\le \min\{e^s,e^{-s}\}$, there is a positive constant $C$ such that
\begin{align*}
&\left|\E\left[R^{\ge 2, blue}_i\Big|W=w\right]-\frac{1}{r}\int_{-\infty}^{\infty} \P(U_{i+1}\ge s\ge U_{i})\left(\mu+\frac{r}{1+e^{-s}}\right) \ ds\right|\\
&\qquad\le \frac{\lambda}{r}\int_{-\infty}^{-\log n-\log(1/w)} \P(U_{i+1}\ge s\ge U_{i})\ ds+\frac{\lambda}{r}\int_{rT_n-\log n-\log(1/w)}^{\infty} \P(U_{i+1}\ge s\ge U_{i})\ ds\\
&\qquad\qquad+\frac{1}{r}\int_{-\log n-\log(1/w)}^{rT_n-\log n-\log(1/w)}\P(U_{i+1}\ge s \ge U_{i})\left|\lambda q(n\delta_{T_n}/ w, t(s))-\mu-\frac{r}{1+e^{-s}}\right|\ ds\\
&\qquad\le \frac{\lambda}{r}\int_{-\infty}^{-\log n-\log(1/w)} e^s\ ds+\frac{\lambda}{r}\int_{rT_n-\log n-\log(1/w)}^{\infty} e^{-s}\ ds\\
&\qquad\qquad+\frac{1}{r}\int_{-\log n-\log(1/w)}^{0}e^{s}\cdot Cr(e^{-rt(s)}+n\delta_{T_n}/ w) \ ds\\
&\qquad\qquad+\frac{1}{r}\int_{0}^{rT_n-\log n-\log(1/w)}e^{-s}\cdot Cr(e^{-rt(s)}+n\delta_{T_n}/ w) \ ds\\
&\qquad\le \frac{\lambda}{r}\cdot\frac{w}{n}+\frac{\lambda}{r}\cdot \frac{n}{we^{rT_n}}+CrT_n\left(\frac{n e^{-rT_n}}{w}+n\delta_{T_n}/ w\right).\\
\end{align*}
By \eqref{Proof of theorems, moment computation, calculus fact}, we have 
\begin{align*}
&\frac{1}{r}\int_{-\infty}^{\infty} \P(U_{i}\ge s\ge U_{i+1})\left(\mu+\frac{r}{1+e^{-s}}\right) \ ds=\frac{1}{r}\int_{-\infty}^{\infty} \frac{e^s}{(1+e^s)^2}\left(\mu+\frac{r}{1+e^{-s}}\right) \ ds=\frac{\mu}{r}+\frac{1}{2}.
\end{align*}
Therefore, under condition \eqref{Introduction: sequence condition, equivalent}, we have
\begin{equation}\label{Proof of theorems, moment computation, blue i}
\lim_{n\rightarrow\infty}\sqrt{n}\left|\E\left[R^{\ge 2, blue}_{i}\Big| W=w\right]-\frac{\mu}{r}-\frac{1}{2}\right|=0.
\end{equation}
Combining \eqref{Proof of theorems, moment computation, red i} and \eqref{Proof of theorems, moment computation, blue i}, we have
$$
\lim_{n\rightarrow\infty}\sqrt{n}\left|\E\left[R^{\ge 2}_i\Big|W=w\right]-\frac{\lambda}{r}\right|=0.
$$
Then \eqref{Proof of theorems, moment computation, R^>=2 mean} follows from summing over $i$.

Equation \eqref{Proof of theorems, moment computation, R^k mean} can be proved analogously. However, since we are only concerned about the leading term, applying the dominated convergence theorem makes the argument easier. We use the change of variable $t=T_n-(s+\log n+\log(1/w))/r$, and then
\begin{align*}
&\E\left[R^{k}_i\Big|W=w\right]\\ 
&\qquad=\E\left[\Pi_i\left(\mathcal{A}(\lambda q(n\delta_{T_n}/ w, \cdot), \left[\max_{i+1\le j\le i+k-1} H_j, H_{i}\wedge H_{i+k}\right]\right)\Big|W=w\right]\\
&\qquad\qquad+\P\left(\max_{i+1\le j\le i+k-1} H_j\le  H_{i+k}\le H_{i}\Big|W=w\right)\\
&\qquad= \int_0^{T_n} \P\left(\max_{i+1\le j\le i+k-1} H_j\le t\le H_{i}\wedge H_{i+k}\Big|W=w\right) \lambda q(n\delta_{T_n}/ w,t)\ dt+\frac{(k-1)!}{(k+1)!}\\
&\qquad=\frac{1}{r}\int_{-\log n-\log(1/w)}^{rT_n-\log n-\log(1/w)} \P\left(\min_{i+1\le j\le i+k-1} U_{j}\ge s\ge U_{i}\vee U_{i+k}\right)\\
&\qquad\qquad \lambda q(n\delta_{T_n}/ w,T_n-(s+\log n+\log(1/w))/r) \ ds+\frac{1}{k(k+1)}.
\end{align*}
As $n\rightarrow\infty$, we have ${-\log n-\log(1/w)}\rightarrow-\infty$, $rT_n-\log n-\log(1/w)\rightarrow\infty$, and $\lambda q(n\delta_{T_n}/ w, t(s))\newline\rightarrow\mu+r/(1+e^{-s})$. Note also that 
\begin{align*}
&\P\left(\min_{i+1\le j\le i+k-1} U_{j}\ge s\ge U_{i}\vee U_{i+k}\right) \lambda q(n\delta_{T_n}/ w, t(s))\le \lambda \P\left(\min_{i+1\le j\le i+k-1}U_{j}\ge s\ge U_{i}\vee U_{i+k}\right), 
\end{align*}
which is integrable over $\R$. Therefore, by the dominated convergence theorem and equation \eqref{Proof of theorems, moment computation, calculus fact}, we have
\begin{align}\label{Proof of theorems, moment computation, R^k_i mean}
&\lim_{n\rightarrow\infty} \E\left[R^{k}_i\Big|W=w\right]\nonumber\\ 
&\qquad=\frac{1}{r}\int_{-\infty}^{\infty} \P\left(\min_{i+1\le j\le i+k-1} U_{j}\ge s\ge U_{i}\vee U_{i+k}\right) \cdot \left(\mu+\frac{r}{1+e^{-s}}\right)\ ds+\frac{1}{k(k+1)}\nonumber\\
&\qquad=\frac{1}{r}\int_{-\infty}^{\infty} \frac{e^{2s}}{(1+e^s)^{k+1}} \cdot \left(\mu+\frac{r}{1+e^{-s}}\right)\ ds+\frac{1}{k(k+1)}\nonumber\\
&\qquad=\frac{\mu}{r}\int_{-\infty}^{\infty} \frac{e^{2s}}{(1+e^s)^{k+1}}\ ds+\int_{-\infty}^{\infty} \frac{e^{3s}}{(1+e^s)^{k+2}}\ ds +\frac{1}{k(k+1)}\nonumber\\
&\qquad=\frac{\lambda-r}{r}\left(-\frac{1}{k}+\frac{1}{k-1}\right)+\left(\frac{1}{k+1}-\frac{2}{k}+\frac{1}{k-1}\right)+\frac{1}{k(k+1)}\nonumber\\
&\qquad=\frac{\lambda}{rk(k-1)}.    
\end{align}

To prove equation \eqref{Proof of theorems, moment computation, R^>=2 variance}, note that $R^{\ge 2}_{i}$ and $R^{\ge 2}_{j}$ are conditionally independent given $W=w$ if $|i-j|\ge 2$. It follows that
\begin{equation}\label{Proof of theorems, moment computation, variance in terms of covariances}
\lim_{n\rightarrow\infty}n^{-1}\Var\left(\sum_{i=1}^{n-2}R^{\ge 2}_{i}\Big|W=w\right)=\lim_{n\rightarrow\infty}\left(\Var\left(R^{\ge 2}_i\Big|W=w\right)+2\Cov\left(R^{\ge 2}_{i}, R^{\ge 2}_{i+1}\Big|W=w\right)\right).   
\end{equation}
To compute $\Var(R^{\ge 2}_i)$, since $R^{\ge 2, red}_{i}$ is $\{0,1\}$-valued, we have
\begin{equation}\label{Proof of theorems, moment computation, red i, red i}
\E\left[\left(R^{\ge 2,red}_{i}\right)^2\Big|W=w\right] = \E\left[R^{\ge 2,red}_{i}\Big|W=w\right] = \frac{1}{2}.   
\end{equation}
By the definition, $R^{\ge 2,blue}_{i}$ is nonzero only if the $\{0,1\}$-valued random variable $R^{\ge 2, red}_{i}$ is nonzero. Then, it follows from equation \eqref{Proof of theorems, moment computation, blue i} that
\begin{equation}\label{Proof of theorems, moment computation, red i, blue i}
\lim_{n\rightarrow\infty}\E\left[R^{\ge 2, red}_{i} R^{\ge 2, blue}_{i}\Big|W=w\right]=\lim_{n\rightarrow\infty}\E\left[ R^{\ge 2, blue}_{i}\Big|W=w\right]=\frac{\mu}{r}+\frac{1}{2}.
\end{equation}
Recall that \textbf{Leb} denotes the Lebesgue measure on $\R^2$ and $t(s)=T_n-(s+\log n+\log (1/w))/r$. A computation similar to the computation for $\E[R^k_{i}|W=w]$ gives the following
\begin{align*}
&\E\left[\left(R^{\ge 2, blue}_i\right)^2\Big|W=w\right]\\ 
&\qquad=\E\left[\Pi_i\left(\mathcal{A}(\lambda q(n\delta_{T_n}/ w, \cdot), \left[H_{i+1}, H_{i}\right]\right)^2\Big|W=w\right]\\
&\qquad=\E\left[\textbf{Leb}\left(\mathcal{A}(\lambda q(n\delta_{T_n}/ w, \cdot), \left[H_{i+1}, H_{i}\right])\right)^2+\textbf{Leb}\left(\mathcal{A}(\lambda q(n\delta_{T_n}/ w, \cdot), \left[H_{i+1}, H_{i}\right])\right)\Big|W=w\right]\\
&\qquad=\int_0^{T_n} \P\left(H_{i+1}\le t_1\le H_{i},H_{i+1}\le t_2\le H_{i}\Big|W=w\right)\lambda q(n\delta_{T_n}/ w,t_1)\lambda q(n\delta_{T_n}/ w,t_2)\ dt_1\ dt_2\\
&\qquad\qquad+\E\left[R_i^{\ge 2, blue}|W=w\right]\\
&\qquad=\frac{1}{r^2}\int_{-\log n-\log(1/w)}^{rT_n-\log n-\log(1/w)} \P\left(U_{i+1}\ge s_1\ge U_{i},U_{i+1}\ge s_2\ge U_{i}\right)\nonumber\\
&\qquad\qquad\lambda q(n\delta_{T_n}/ w,t(s_1))\lambda q(n\delta_{T_n}/ w,t(s_2)) \ ds_1\ ds_2+\E\left[R_i^{\ge 2, blue}|W=w\right].
\end{align*}
We apply the dominated convergence theorem as $n\rightarrow\infty$, use equation \eqref{Proof of theorems, moment computation, blue i}, and evaluate the last integral using a computer to get
\begin{align}\label{Proof of theorems, moment computation, blue i, blue i}
&\lim_{n\rightarrow\infty}\E\left[\left(R^{\ge 2, blue}_i\right)^2\Big|W=w\right]\nonumber\\
&\qquad=\frac{1}{r^2}\int_{-\infty}^{\infty}\int_{-\infty}^{\infty} \P(U_i\ge s_1\ge U_{i+1}, U_i\ge s_2\ge U_{i+1}) \nonumber\\
&\qquad\qquad\qquad\qquad\left(\mu+\frac{r}{1+e^{-s_1}}\right)\left(\mu+\frac{r}{1+e^{-s_2}}\right)\ ds_1\ ds_2+\frac{\mu}{r}+\frac{1}{2}\nonumber\\ 
&\qquad=\frac{2}{r^2}\int_{-\infty}^{\infty}\int_{-\infty}^{s_1} \frac{1}{1+e^{s_1}}\frac{e^{s_2}}{1+e^{s_2}} \left(\mu+\frac{r}{1+e^{-s_1}}\right)\left(\mu+\frac{r}{1+e^{-s_2}}\right)\ ds_2\ ds_1+\frac{\mu}{r}+\frac{1}{2}\nonumber\\
&\qquad=\frac{2}{r^2}\left(\frac{\pi^2}{6}\mu^2+\frac{\pi^2}{6}\mu r+\frac{1}{2}r^2\right)+\frac{\mu}{r}+\frac{1}{2}.
\end{align}
Combining \eqref{Proof of theorems, moment computation, R^>=2 mean}, \eqref{Proof of theorems, moment computation, red i, red i}, \eqref{Proof of theorems, moment computation, red i, blue i} and \eqref{Proof of theorems, moment computation, blue i, blue i}, we have
\begin{align}\label{Proof of theorems, moment computation, i, i}
\lim_{n\rightarrow\infty} \Var\left(R^{\ge 2}_i|W=w\right) &= \frac{1}{2}+2\left(\frac{\mu}{r}+\frac{1}{2}\right)+\frac{2}{r^2}\left(\frac{\pi^2}{6}\mu^2+\frac{\pi^2}{6}\mu r+\frac{1}{2}r^2\right)+\frac{\mu}{r}+\frac{1}{2}-\left(\frac{\lambda}{r}\right)^2\nonumber\\
&=\left(\frac{\pi^2}{3}-1\right)\frac{\mu^2}{r^2}+\left(\frac{\pi^2}{3}+1\right)\frac{\mu}{r}+2. 
\end{align}
The computation for $\Cov(R^{\ge 2}_i, R^{\ge 2}_{i+1})$ is analogous. We have
\begin{equation}\label{Proof of theorems, moment computation, red i, red i+1}
\E\left[R^{\ge 2,red}_{i}R^{\ge 2,red}_{i+1}\Big|W=w\right]=\P(H_{i+2}\le H_{i+1}\le H_{i}|W=w)=\frac{1}{6},   
\end{equation}
\begin{align}\label{Proof of theorems, moment computation, red i, blue i+1}
\lim_{n\rightarrow\infty} \E\left[R^{\ge 2,red}_i R^{\ge 2,blue}_{i+1}\Big|W=w\right] 
&= \frac{1}{r}\int_{-\infty}^{\infty} \P(U_i\le U_{i+1}\le s\le U_{i+2}) \cdot \left(\mu+\frac{r}{1+e^{-s}}\right)\ ds\nonumber\\
&= \frac{1}{r}\int_{-\infty}^{\infty} \frac{e^{2s}}{2(1+e^s)^3} \cdot \left(\mu+\frac{r}{1+e^{-s}}\right)\ ds\nonumber\\
&=\frac{\mu}{4r}+\frac{1}{6},   
\end{align}
\begin{align}\label{Proof of theorems, moment computation, blue i, red i+1}
\lim_{n\rightarrow\infty} \E\left[R^{\ge 2,blue}_i R^{\ge 2,red}_{i+1}\Big|W=w\right] 
&= \frac{1}{r}\int_{-\infty}^{\infty} \P(U_i\le s\le U_{i+1}\le  U_{i+2}) \cdot \left(\mu+\frac{r}{1+e^{-s}}\right)\ ds\nonumber\\
&= \frac{1}{r}\int_{-\infty}^{\infty} \frac{e^{s}}{2(1+e^s)^3} \cdot \left(\mu+\frac{r}{1+e^{-s}}\right)\ ds\nonumber\\
&=\frac{\mu}{4r}+\frac{1}{12}, 
\end{align}
and evaluating the last integral using a computer,
\begin{align}\label{Proof of theorems, moment computation, blue i, blue i+1}
&\lim_{n\rightarrow\infty}\E\left[R^{\ge 2, blue}_i R^{\ge 2,blue}_{i+1}\Big| W=w\right]\nonumber\\
&=\frac{1}{r^2}\int_{-\infty}^{\infty}\int_{-\infty}^{\infty} \P(U_i\le s_1\le U_{i+1}\le s_2\le U_{i+2}) \cdot \left(\mu+\frac{r}{1+e^{-s_1}}\right)\cdot \left(\mu+\frac{r}{1+e^{-s_2}}\right)\ ds_1\ ds_2\nonumber\\
&=\frac{1}{r^2}\int_{-\infty}^{\infty}\int_{-\infty}^{s_1} \frac{1}{1+e^{s_1}}\left(\frac{e^{s_1}}{1+e^{s_1}}-\frac{e^{s_2}}{1+e^{s_2}}\right)\frac{e^{s_2}}{1+e^{s_2}} \cdot \left(\mu+\frac{r}{1+e^{-s_1}}\right)\cdot \left(\mu+\frac{r}{1+e^{-s_2}}\right)\ ds_2\ ds_1\nonumber\\
&=\frac{1}{r^2}\left(\left(2-\frac{\pi^2}{6}\right)\mu^2+\left(2-\frac{\pi^2}{6}\right)\mu r+\frac{1}{12}r^2\right).
\end{align} 
Combining \eqref{Proof of theorems, moment computation, red i}, \eqref{Proof of theorems, moment computation, blue i}, \eqref{Proof of theorems, moment computation, red i, red i+1}, \eqref{Proof of theorems, moment computation, red i, blue i+1}, \eqref{Proof of theorems, moment computation, blue i, red i+1} and \eqref{Proof of theorems, moment computation, blue i, blue i+1}, we get
\begin{equation}\label{Proof of theorems, moment computation, i, i+1}
\lim_{n\rightarrow\infty} \Cov(R^{\ge 2}_{i}, R^{\ge 2}_{i+1}|W=w)=\left(1-\frac{\pi^2}{6}\right)\frac{\mu^2}{r^2}+\left(\frac{1}{2}-\frac{\pi^2}{6}\right)\frac{\mu}{r}-\frac{1}{2}.
\end{equation}
Combining \eqref{Proof of theorems, moment computation, variance in terms of covariances}, \eqref{Proof of theorems, moment computation, i, i} and \eqref{Proof of theorems, moment computation, i, i+1}, we have
$$
\lim_{n\rightarrow\infty}n^{-1}\Var\left(\sum_{i=1}^{n-2}R^{\ge 2}_{i}\Big|W=w\right)=\frac{\mu^2}{r^2}+\frac{2\mu}{r}+1=\frac{\lambda^2}{r^2}.
$$
\end{proof}

\subsection{Proof of Theorem \ref{Theorem: LLN}}\label{Subsection: proof of theorem LLN}
By Corollary \ref{Corollary: error bounds}, it suffices to prove Proposition \ref{Proposition: lln}. Conditional on $W=w$, the random sequences $(R^k_i)_{i\le j}$ and $(R^k_{i})_{i\ge j+k+1}$ are independent for all $j$. We can, therefore, apply the law of large numbers to $(R^k_{i+j(k+1)})_{j\ge 1}$ for $i=1,2,\dots,k,k+1$. We have
$$
\frac{\sum_{j=1} ^{\lfloor(n-i)/(k+1)\rfloor} R_{i+j(k+1)}}{n} \stackrel{P}{\rightarrow}\frac{\lambda}{r k(k-1)(k+1)}.
$$
Then Proposition \ref{Proposition: lln} follows from \eqref{Proof of theorems, moment computation, R^k mean} and recombining the results for $i=1,2,\dots,k,k+1$.

\subsection{Proof of Theorem \ref{Theorem} and Corollary \ref{Corollary: number of mutations}} \label{Subsection: proof of theorem}
\begin{proof}[Proof of Theorem \ref{Theorem}]
By Corollary \ref{Corollary: error bounds}, it suffices to prove Proposition \ref{Proposition: clt}. Recall from the discussion after Proposition \ref{Proposition: clt} that it suffices to prove the central limit theorem conditional on $W$. Given $W=w$, the random sequences $(R^{\ge 2}_{i})_{i\le k}$ and $(R^{\ge 2}_{i})_{i\ge k+2}$ are independent for all $k$. We apply the central limit theorem for $m$-dependent sequences to this sequence. By Theorem 3 of \cite{diananda1955central}, it suffices to check the following conditions:
\begin{equation}\label{Proof of theorems, proof of clt, condition 1}
\liminf_{n\rightarrow\infty} \frac{1}{n}\Var\left(\sum_{i=1}^{n-2}R^{\ge 2}_i\Big|W=w\right)>0, 
\end{equation}
\begin{equation}\label{Proof of theorems, proof of clt, condition 2} 
\lim_{n\rightarrow\infty}\frac{1}{n} \sum_{i=1}^{n-2} \E\left[\left|R^{\ge 2}_{i}-\E\left[R^{\ge 2}_{i}\right]\right|^2\mathbbm{1}_{\left\{\left|R^{\ge 2}_{i}-\E\left[R^{\ge 2}_{i}\right]\right|>\epsilon\sqrt{n}\right\}}\Big|W=w\right]=0,\quad \forall \epsilon>0.  
\end{equation}
If \eqref{Proof of theorems, proof of clt, condition 1} and \eqref{Proof of theorems, proof of clt, condition 2} hold, then Proposition \ref{Proposition: clt} follows from \eqref{Proof of theorems, moment computation, R^>=2 variance} and \eqref{Proof of theorems, moment computation, R^>=2 mean}. Equation \eqref{Proof of theorems, proof of clt, condition 1} follows from Lemma \ref{Lemma: moment computation}. For equation \eqref{Proof of theorems, proof of clt, condition 2}, it is sufficient to prove the following bound; see p.362 of \cite{billingsley1995probability};
$$
\sup_{n}\E\left[\left(R_{i}^{\ge 2}\right)^p\Big|W=w\right]<\infty, \qquad\text{ for all } p\in(2,\infty). 
$$
Recall that 
\begin{align*}
R^{\ge 2}_{i}&=R^{\ge 2,blue}_{i}+R^{\ge 2,red}_{i}\\
&=\Pi_i(\mathcal{A}(\lambda q(n\delta_T/ W,\cdot),[H_{i+1},H_{i}]))+\mathbbm{1}_{\{H_{i+1}\le H_{i}\}}\\
&\le \Pi_i(\mathcal{A}(\lambda,[H_{i+1},H_{i}]))+1\\
&=\Pi_i(\mathcal{A}(\lambda,[T-(\log n+\log(1/W)+U_i)/r,T-(\log n+\log(1/W)+U_{i+1})/r]))+1.
\end{align*}
Conditional on $W=w$, the right hand side is distributed as $\Pi_i(\mathcal{A}(\lambda,[-U_i/r,-U_{i+1}/r]))+1$, a random variable that does not depend on $n$ and has finite $p$th momemnt for all $p\in(2,+\infty)$.
\end{proof}

\begin{proof}[Sketch of proof of Corollary \ref{Corollary: number of mutations}]
The proof of this corollary is similar to the proof of Proposition 11 in the supplementary information in \cite{johnson2023clonerate}. Following the notation therein, we decompose the fluctuation in $M^{\ge 2}_{n,T_n}$ into the fluctuation in the number of reproduction events $R^{\ge 2}_{n,T_n}$ and the fluctuation of $M^{\ge 2}_{n,T_n}$ given $R^{\ge 2}_{n,T_n}$. Let
$$
A_n =\frac{M^{\ge 2}_{n,T_n}-\nu R^{\ge 2}_{n,T_n}}{\sqrt{n\lambda \nu /r}},\qquad C_n =\frac{1}{\sqrt{n}}\left(R_{n,T_n}^{\ge 2}-\frac{n\lambda}{r}\right),
$$
so that
$$
M^{\ge 2}_{n,T_n}=\frac{n\lambda\nu}{r}+\sqrt{\frac{n\lambda \nu }{r}}A_n+\sqrt{n}\nu C_n.
$$
Then $(A_n, C_n)\Rightarrow(Z_1,Z_2)$ where $Z_1,Z_2$ are independent normal random variables with mean $0$ and variances $1$ and $\lambda^2/r^2$. The $\sqrt{n\lambda \nu /r} A_n$ term contributes the $\lambda\nu/r$ term of the variance in Corollary \ref{Corollary: number of mutations} and the $\sqrt{n}\nu C_n$ term contributes the other part of the variance. 
\end{proof}

\noindent {\bf {\Large Acknowledgments}}

\bigskip
\noindent The author thanks Professor Jason Schweinsberg for bringing up this research topic, carefully reading the draft, and providing helpful advice in the development of this paper.\

\bibliographystyle{plain}
\bibliography{Ref}
\end{document}